\documentclass[letterpaper, 10 pt, conference]{ieeeconf}  % Comment this line out if you need a4paper

\IEEEoverridecommandlockouts                              % This command is only needed if 
\usepackage[utf8]{inputenc}
\usepackage{chngcntr}
\usepackage{apptools}
\AtAppendix{\counterwithin{theorem}{subsection}}

\usepackage{tkz-graph}

\usepackage{cancel}
\usepackage[normalem]{ulem}
\usepackage[T1]{fontenc} % 8-bit encoding
\usepackage{color}
\usepackage[leqno]{amsmath}

\usepackage{amsthm}
\usepackage{amsfonts, amssymb, commath} % Math packages
\usepackage{mathtools}
\usepackage{graphicx}
\usepackage{mathrsfs}
\usepackage{caption}
\usepackage{subcaption}
\usepackage{listings}
\usepackage[]{eufrak}
\usepackage[]{hyperref}
\usepackage[]{float}
\usepackage{tikz}
\usetikzlibrary{calc, arrows,matrix,positioning,scopes, intersections}
\usetikzlibrary{shapes.geometric}
\usetikzlibrary{decorations.markings}
\usepackage{tikz-cd}
\usetikzlibrary{cd}
\usepackage{booktabs}
\usepackage{enumerate}
\usepackage{eufrak}
\usepackage{MnSymbol}
\usepackage{tikz}
\usepackage{algorithm,algpseudocode}
\usetikzlibrary{matrix,decorations.pathreplacing,calc}
\pgfkeys{tikz/mymatrixenv/.style={decoration=brace,every left delimiter/.style={xshift=3pt},every right delimiter/.style={xshift=-3pt}}}
\pgfkeys{tikz/mymatrix/.style={matrix of math nodes,left delimiter={(}, right delimiter={)}, inner sep=1pt,column sep=1.5em,row sep=0.5em,nodes={inner sep=0pt}}}
\pgfkeys{tikz/mymatrixbrace/.style={decorate,thick}}

\usepackage{cuted}
\newcommand{\Tr}{\mathop{\bf Tr}}
\newcommand{\ca}[1]{\mathcal{#1}}

\newcommand{\bb}[1]{\mathbb{#1}}

\newcommand{\rank}{\text{\bf rank}}

\newcommand{\argmin}{\mathop{\rm arg\,min}}

\makeatletter
\newcommand{\leqnomode}{\tagsleft@true}
\newcommand{\reqnomode}{\tagsleft@false}
\makeatother
\renewcommand{\theenumi}{\alph{enumi}}

\newtheorem{theorem}{Theorem}[section]

\newtheorem{lemma}[theorem]{Lemma}
\newtheorem{proposition}[theorem]{Proposition}

\newtheorem{remark}[theorem]{Remark}

\tikzset{%
    add/.style args={#1 and #2}{
        to path={%
 ($(\tikztostart)!-#1!(\tikztotarget)$)--($(\tikztotarget)!-#2!(\tikztostart)$)%
  \tikztonodes},add/.default={.2 and .2}}
}
\newtheorem*{assumption*}{\assumptionnumber}
\providecommand{\assumptionnumber}{}
\makeatletter
\newenvironment{assumption}[1]
 {%
  \renewcommand{\assumptionnumber}{Assumption #1}%
  \begin{assumption*}%
  \protected@edef\@currentlabel{#1}%
 }
 {%
  \end{assumption*}
 }
\makeatother

\usepackage{cite}
\title{Global Convergence of Policy Gradient Algorithms for\\Indefinite Least Squares Stationary Optimal Control}
\author{Jingjing Bu and Mehran Mesbahi
\thanks{Submitted to IEEE Control Systems Letters. The research of the authors has been supported by AFOSR 
Grant No. FA9550-16-1-0022 and DARPA Lagrange Grant FA8650-18-2-7836.}
\thanks{The authors are with the University of Washington, Seattle, WA 98195, USA; Emails: bujing+mesbahi@uw.edu}
}

\begin{document}

\maketitle
\thispagestyle{empty}
\pagestyle{empty}
\begin{abstract}
  We consider policy gradient algorithms for the 
  indefinite least squares stationary optimal control, 
  e.g., linear-quadratic-regulator (LQR) with indefinite state and input penalization matrices. Such a setup has important applications in control design with conflicting objectives, such as linear quadratic dynamic games. We show the global convergence of gradient, natural gradient and quasi-Newton policies for this class of indefinite least squares problems.
\end{abstract}
%\begin{IEEEkeywords}
  %First-order methods; Linear quadratic control; Linear quadratic games; Policy iteration; Reinforcement Learning
%\end{IEEEkeywords}
%
\section{Introduction}
\label{sec:intro}
Least squares stationary optimal control provides an effective synthesis procedure for linear control systems
since Kalman's original work in the 1960s~\cite{kalman1960contributions}. This setting was later extended beyond positive semidefinite cost structure by Willems~\cite{willems1971least}. It is known that similar to standard LQR, this setup can be examined using the Algebraic Riccati Equation (ARE); DARE refers to the discrete analogue of this matrix equation.
%LQR is formulated around an optimization problem for determining a sequence of (control) inputs to a linear system in order to minimize a given (integral) quadratic cost over an infinite horizon.\footnote{We shall not delve into the finite-horizon LQR in this paper.}
%% while imposing constraints on the motion of the system. 
%From the theoretical point of view, a fundamental property of LQR synthesis is that the resulting optimal input is in the form of a state feedback; as such, it can be represented as a constant feedback gain on the state of the system~\cite{Kalman_BSMM_1960,Anderson_book_1990}.
%%
%The state feedback gain that ``solves'' the infinite-horizon LQR problem, in turn, can be obtained by solving the algebraic Riccati equation (ARE). That is, in the traditional approach to LQR design, the state feedback gain is revealed after obtaining the ``certificate'' or ``cost-to-go'' for the underlying optimal control problem.\footnote{The analogy here would be solving the dual, followed by the recovery of the primal solution.}
%
Historically, a large number of works have studied the solution of ARE and DARE, including approaches based on iterative algorithms~\cite{Hewer1971TAC},\footnote{In Hewer's original work, $Q$ and $R$ are positive definite. However, the algorithm still converges even for the indefinite cost structure~\cite{Lancaster1995algebraic}.} algebraic solution methods~\cite{Lancaster1995algebraic}, and semidefinite programming~\cite{Balakrishnan2003TAC}.

Although the cost function plays a fundamental role in the least squares optimal control, it is generally not ``recommended'' to {\em directly} compute the optimal gain (policy) using this cost without solving the associated Riccati equation.\footnote{In this note, feedback gain, feedback control and feedback policy are used interchangeably.}  This approach, in the meantime, is in sharp contrast to how one would typically go about minimizing a cost function over the variable of interest in introductory optimization, say, through gradient descent.\footnote{This is essentially due to the dynamic nature of the constraint set.}
% With recent advances in  sophisticated statistical and optimization methods,
Recently, there has been a surge of interest in constructing optimal control strategies directly, viewing control synthesis through the lens of first order methods.\footnote{One might as well extrapolate that these methods provide a streamline recipe for learning optimal feedback gains in real-time.}
%from optimizing the cost function. 
Adopting such a point of view has been partially inspired by the application of
learning algorithms in control, such as Reinforcement Learning (RL), where 
using principles of (approximate) dynamic programming, one can devise 
real-time model-free methods for both continuous-time and discrete-time optimal control problems~\cite{Jiang2012Auto,Young2012Auto,Lee2019TAC,Bradtke1994ACC,Lewis2009CSM,lewis2012reinforcement,Chun2016IJC,bertsekas2005dynamic}. 
%It has been noted that LQR problem can be formulated as RL problem~\cite{bertsekas2005dynamic}. 
The RL perspective not only provides more insights into the synthesis problem, but also can be extended to model-free settings by means of stochastic (zeroth-order) optimization~\cite{conn2009introduction, nemirovski2009robust}.
 However, policy iteration is inherently prohibitive for an infinite horizon cost structure that is undiscounted and unbounded per stage~\cite{bertsekas2005dynamic}. \par
The main contribution of this note is to extend policy based algorithms beyond positive (semi)definite cost structures considered in~\cite{fazel2018global, bu2019lqr}. More specifically, we show that under mild assumptions, even when the state and cost penalization matrices are indefinite in the least squares optimal control, gradient policy (respectively, natural gradient and quasi-Newton policies) converges to the global minimizer at a linear (respectively, linearly and $Q$-quadratic) rate. Along the way, we devise a distinct approach for arguing the stability of the iterative process as compared with those adopted in previous works.\footnote{The proposed technique also provides an alternative way to argue stability properties of the iterative process under standard LQR assumptions.}\par 
The note is organized as follows. In \S\ref{sec:notation}, we introduce the notation and preliminaries. \S\ref{sec:dlqr} is devoted to the LQR setup, analytical properties of the cost function, a ``mild'' assumption, and its implications. In \S\ref{sec:convergence}, we derive the corresponding stepsizes for gradient descent (GD), natural gradient descent (NGD), and quasi-Newton (QN) iterations; we then show the global linear (respectively, linear and $Q$-quadratic) convergence of gradient policy (respectively, natural gradient policy and quasi-Newton policy) under the proposed stepsizes. A numerical example is provided in \S\ref{sec:numerical}. The note is concluded in \S\ref{sec:conclusion}. 
\section{Notation and Preliminaries}
\label{sec:notation}We denote by ${\bb M}_{n \times m}(\bb R)$ the set of $n \times m$ real matrices. $\bb R^n$ denotes the $n$-dimensional real Euclidean space; when $n=1$, this set is identified with the set of real numbers. Other notation include $A^\top$, $\rho(A)$, $\Tr (A)$, representing the transpose, spectral radius, and trace of the matrix $A$, respectively.
%and vectorization of the matrix $A$, respectively;  $A \otimes B$ is the Kronecker product of matrices $A$ and $B$.
The real inner product between a pair of vectors $x$ and $y$ is denoted by $\langle x,y \rangle$. $\|A\|_2$ denotes the spectral (operator) norm of a square matrix $A$ and $\|A\|_{F}$ denotes its Frobenius norm.\footnote{2-norm is assumed when we use $\|.\|$.}
Lastly, the notation $A \succeq B$ for two symmetric matrices refers to the positive semi-definiteness of their difference $A-B$; analogously for positive definiteness of this difference we use $A \succ B$. We let $\lambda_i(A)$ denote the eigenvalues of a square matrix $A$. These eigenvalues are indexed in an increasing order with respect to their real parts, i.e.,
\begin{align*}
  {\bf Re}(\lambda_1(A)) \le \dots \le {\bf Re}(\lambda_n(A)).
  \end{align*}
  If $A$ is symmetric, the ordering becomes $\lambda_1(A) \le \dots \le \lambda_n(A)$. 
%  $\|A\|$ shall be used to denote the operator norm (largest signular value) of $A$ and $\|A\|_F$ denotes the Frobenius norm of a matrix. 
  When $A \succeq 0$, $\|A\| = \lambda_n(A)$ and we shall use these interchangeably.
%\par
We use $C^{\omega}(U)$ to denote the set of real analytic functions over an open set $U \subseteq \bb R^n$. %A function $f$ is $L$-smooth when $f$ is \emph{continuously differentiable} and its gradient is $L$-Lipschitz, i.e., $\|\nabla f(x)-\nabla f(y)\| \le L \|x-y\|$.
%\par
% Notions from algebraic geometry are used in proofs to Proposition~\ref{prop:eigen_components}, please consult~\cite{liu2006algebraic} for background information. \par
% We shall differentiate asymptotic convergence rate and nonasymptotic convergence rate. A convergent sequence $\{x_k\}$ in a normed vector converges to $x^*$ at a nonasymptotic linear rate $O(q^k)$ if $\|x_k - x^*\| \le q \|x_{k-1} - x^*\|$ for every $k \in \bb N$. The sequence converges at an asymptotic linear rate $\ca O(q^k)$ if $\lim_{k \to \infty} \frac{\|x_k - x^*\|}{\|x_{k-1} - x^*} = q$. In following, $O(q^k)$ and $\ca O(q^k)$ should be interpreted in the precise sense. \par
A square matrix $A \in \bb M_{n \times n}(\bb R)$ is \emph{Schur} if $\rho(A) < 1$.
%A pair $(A, B)$ with $A \in {\bb M}_{n \times n}(\bb R)$ and $B \in {\bb M}_{n \times m}( \bb R)$ is called controllable if the Kalman rank condition~\cite{sontag2013mathematical},
%\begin{align*}
  %\rank([B, AB, A^2B, \dots, A^{n-1}B]) = n,
%\end{align*}
%is satisfied.
A pair $(A, B)$ is stabilizable if there exists some $K$ for which $A-BK$ is Schur.
Given a pair of system matrices $(A, B)$, $\ca S$ denotes the set of Schur stabilizing feedback gains, 
$$\ca S =\{K \in {\bb M}_{m \times n}(\bb R): \rho(A-BK) < 1\}.$$
For the pair $(A, B)$, we say that $K$ is stabilizing if $A-BK$ is Schur; it is \emph{marginally stabilizing} or \emph{almost stabilizing} if $\rho(A-BK) = 1$. An eigenvalue $\lambda$ of $A \in {\bb M}_{n \times n}(\bb R)$ is called $(C, A)$-observable if $$\rank \left( \begin{pmatrix} A - \lambda I \\ C\end{pmatrix}\right) = n,$$
for a given $C \in {\bb M}_{p \times n}(\bb R)$; $p$ is the dimension of the output of a linear system.

\section{Problem Setup}
\label{sec:dlqr}
In the standard least squares (stationary) optimal control, we consider a (discrete-time) linear time invariant model of the form,
\begin{align}
  x_{k+1} = Ax_k + B u_k, \label{LTI}
\end{align}
where $A \in {\bb M}_{n \times n}( \bb R)$, $B \in  {\bb M}_{n \times m}( \bb R)$ and $(A, B)$ is stabilizable. The corresponding LQR problem is the optimization problem of devising a linear feedback gain $K \in {\bb M}_{m \times n} (\bb R)$ for which $u_k = -K x_k$, minimizing,\footnote{The condition that  $u_k$ has the form $-K x_k$ is not set a priori in the LQR formulation; this feedback form is typically shown via the adoption of a dynamic programming step.}
\begin{align*}
  J (x_0) &= \sum_{k=0}^{\infty} \left[ \langle x_k, Q x_k\rangle  + \langle u_k, R u_k\rangle \right],
\end{align*}
where $x_0$ is the initial condition, and the quadratic cost is parameterized by 
$Q = Q^{\top}$ and $R = R^{\top}$; note that we \emph{do not} require positive (semi-)definiteness of $Q$ and $R$. Such a generalization is not only of theoretical interest but also has important applications in network synthesis and stability theory~\cite{willems1971least}.
In order to update the feedback gain (policy) directly, it will conceptually be appealing to consider the cost as a matrix function over the set of feedback gains. With this aim in mind, we may define $J_{x_0} \colon {\bb M}_{m \times n}(\bb R) \to \bb R$ as,
\begin{equation}
  \label{eq:naive_cost_function}
  \resizebox{0.92\linewidth}{!}{$
    J_{x_0}( K) = \sum_{j=0}^{\infty}  \left[ \langle (A-BK)^j x_0, (Q+K^{\top}RK)(A-BK)^j x_0\rangle \right],
                  $}
\end{equation}
for some fixed initial condition $x_0 \in \bb R^n$. Note that the cost function $J$ is a function of the policy $K$ and initial condition $x_0$. Since we are interested in \emph{optimal policy} independent of initial conditions, naturally, we should reformulate the cost function to reflect this independence. Indeed, this point has been discussed in~\cite{bu2019lqr} where it is argued that such a formulation is  necessary for the cost function to be well defined (see details in \S III~\cite{bu2019lqr}). The independence with respect to the initial condition can be achieved by either sampling $x_0$ from a distribution with full-rank covariance~\cite{fazel2018global}, or choosing a spanning set $\{z_1, \dots, z_n\} \subseteq \bb R^n$~\cite{bu2019lqr}, and defining the value function over $\ca S$ as,
\begin{align}
  \label{f(k)}
  f(K) = \sum_{i=1}^n J_{z_i}(K),
\end{align}
where $J_{z_i}(K)$ is the cost by choosing initial state $x_0$ as $z_i$ and letting $u_k = Kx_k$. Note that over the set $\ca S$, $f$ admits a compact form $  f(K) = \Tr( X {\bf \Sigma})$,
where $ {\bf \Sigma} = \sum_{i=1}^n z_i z_i^{\top}$ and $X$ is the solution to the Lyapunov equation,
\begin{align}
  \label{eq:lyapunov_matrix}
  (A-BK)^{\top}X(A-BK) + Q + K^{\top}RK = X.
\end{align}
%Next we present the analytical properties of the value function $f(K, L)$. To begin with, let us first observe some topological properties of the domain of $f(K, L)$, i.e., $\ca S$.
%\paragraph{Behavior of $f$ on $\ca S^{c}$}

How the cost function $f$ behaves near the boundary $\partial \ca S$ is of paramount importance in the design of iterative algorithms for least squares optimal control problems. In the standard setting, the cost function diverges to $+\infty$ when the feedback gain approaches the boundary of this set (see~\cite{bu2019lqr} for details). In fact, this property guarantees stability of the obtained solution via first order iterative algorithms for the suitable choice of stepsize. However, the behavior of $f$ on the boundary $\partial \ca S$ could be more intricate. For example, if $K \in \partial \ca S$, i.e., $\rho(A-B K) = 1$, then it is possible that the cost is still finite. This happens when an eigenvalue of $A-B K$ on the unit disk in the complex plane is not $(Q+K^{\top} R K , A-BK )$-observable. To see this, we note that for every $\omega_i$, the series
\begin{equation*}
  \resizebox{\linewidth}{!}{
  $J_{\omega_i}(K) =  \omega_i^{\top} \left(\sum_{j=0}^{\infty} ((A-BK)^\top)^{j} (Q + K^{\top} R K)(A-B K)^{j} \right)\omega_i$.}
\end{equation*}
is convergent to a finite (real) number if the marginally stable modes are not detectable.
 Even on $\bar{\ca S}^c$ (complement of closure of $\ca S$), $f$ could be finite if all non-stable modes of $A-B K$ are not $(Q+K^{\top} R K , A-BK)$-observable. The complication suggests that the function value is no longer a valid indictor of stability. We remark that such a situation does not occur in the LQ setting examined in~\cite{fazel2018global, bu2019lqr}, as it has been assumed that $Q$ is positive definite.
 \subsection{Analytical properties of the indefinite cost function}
 \label{sec:lqr_analytical}
 In this section, we collect some useful analytic characterizations of $f(K)$. To simplify the notation, in the rest of this paper, we set,%\footnote{If the context is clear, we will drop the subscript of ${\bf N}_K$.}
\begin{align*}
     A_K \coloneqq A-BK,  \quad \mbox{and} \quad {\bf N}_K \coloneqq RK-B^{\top} X(A-BK);
\end{align*}
when the context is clear, we will write ${\bf N}$ instead of ${\bf N}_K$; in describing the iterative process on the gain matrix (when $K$ is updated), we shall denote ${\bf N}_{K_j}$ as ${\bf N}_j$.
\begin{proposition}
  \label{prop:analytical}
  The indefinite least squares optimal control problem~(\ref{f(k)}) on the set of stabilizing feedback gains has the following properties:
  \begin{enumerate}
  \item 
  The set $\ca S$ is regular open, contractible, and unbounded when $m \ge 2$ and the boundary $\partial \ca S$ is precisely the set $\ca B=\{K \in {\bb M}_{m \times n}(\bb R): \rho(A-BK) = 1\}$. 
  \item 
 For the cost $(\ref{f(k)})$, one has $f \in C^{\omega}(\ca S)$.
\item The gradient of $f$~(\ref{f(k)}) is given by
  \begin{align*}
    \nabla f(K) = 2(RK-B^{\top}X A_K) Y_K,
    \end{align*}
    where $Y_K$ solves the Lyapunov matrix equation,
    \begin{align}
      \label{eq:lyapunov_Y}
      A_{K}Y A_K^{\top} + {\bf \Sigma} = Y.
      \end{align}
 \item 
   Let $K, \tilde{K} \in \bar{\ca S}$;\footnote{$\bar{\ca S}$ is the closure of $\ca S$.}; suppose that the corresponding Lyapunov matrix equations~\eqref{eq:lyapunov_matrix} have symmetric solutions $X$ and $\tilde{X}$, respectively.\footnote{Note that the assumption clearly holds if $K, \tilde{K} \in \ca S$. It will also holds if $K \in \partial \ca S$ and the eigenvalues of $A-BK$ on the unit disk are not $(Q+K^{\top}RK, A-BK)$-observable.} Namely,
\begin{align*}
  A_K^{\top}X A_K + Q + K^{\top} R K = X,\\
  A_{\tilde{K}}^{\top}\tilde{X} A_{\tilde{K}} + Q + \tilde{K}^{\top}R \tilde{K} = \tilde{X}.
  \end{align*}
   Then we have
      \begin{align*}
         & A_{\tilde{K}}^{\top} (X-\tilde{X}) A_{\tilde{K}} + (K-\tilde{K})^{\top}{\bf N}_K + {\bf N}_K^{\top}(K-\tilde{K})\\
                      &\quad  - (K-\tilde{K})^{\top}(R+B^{\top} X B)(K-\tilde{K})=X - \tilde{X}.
        \end{align*}
      \item Suppose that $K_* \in \argmin_{K \in \ca S} f(K)$. Then
        \begin{align*}
          \tau_1 \|K-K_*\|_F^2 \le f(K) - f(K_*) \le \tau_2 \langle {\bf N}_K, {\bf N}_K\rangle,
        \end{align*}
        where
        \begin{align*}\tau_1 = \lambda_1(Y)\lambda_1(R+B^{\top}XB), \;
  \tau_2 = \frac{\lambda_n(Y_*)}{\lambda_1(R+B^{\top}X B)}, 
  \end{align*}
  and $Y_*$ solves the Lyapunov equation~\eqref{eq:lyapunov_Y} with $K_*$.
        \end{enumerate}
\end{proposition}
The proofs of these results can be found in~\cite{bu2019lqr}. We emphasize that ($e$) holds only if $\argmin_{K \in \ca S} f(K) \neq \emptyset$, namely, there exists $K_* \in \ca S$ such that $f(K) \ge f(K_*)$ for every $K \in \ca S$. In the next subsection, we shall elaborate on a ``mild'' assumption to ensure that this condition holds. 
 \subsection{A key assumption and its consequences}
Throughout the manuscript, we have the following standing assumption.
\begin{assumption}{1}
  \label{assump1}
  There exists a \emph{strict local minimizer} of $f(K)$ over $\ca S$. In other words, there exists some $K_* \in \ca S$ and an open neighborhood $B_{\delta}(K_*) = \{K : \|K-K_*\|_F < \delta\}$, such that $f(K_*) < f(K)$ for every $K \in B_{\delta}(K_*) \cap \ca S$.
\end{assumption}
\begin{remark}
  The seminal work of Willems~\cite{willems1971least} explores many facets of the least squares optimal control with indefinite $Q$ and $R$;\footnote{An our adopted terminology is in his honor.} in particular, this work examines conditions for which the above assumption holds. We will not discuss these conditions and instead refer the reader to~\cite{willems1971least} and references therein.
\end{remark}
We observe several implications of this assumption.
\begin{proposition}
  Suppose that $K_*$ is the strict local minimizer of $f(K)$ over $\ca S$ and $X_*$ is the corresponding value matrix.
  Then, 
\begin{enumerate}
  \item $X_* = X_*^{\top}$,
  \item $R + B^{\top} X_* B \succ 0$,
    \item $X_*$ solves the DARE~\eqref{eq:dare},
      \begin{equation}
        \label{eq:dare}
    X = A^{\top}XA + Q - A^{\top}XB(R+B^{\top}XB)^{-1}B^{\top}XA,
        \end{equation}
  \item The minimizer $K_*$ is the unique global minimizer,
  \item $X_*$ is the maximal solution to DARE~\eqref{eq:dare} and is unique among all \emph{almost stabilizing solutions} of~\eqref{eq:dare}.
  \end{enumerate}
\end{proposition}
\begin{proof}
Part $(a)$ follows from having $X_*$ solve the Lyapunov matrix equation~\eqref{eq:lyapunov_matrix} with $K = K_*$ and the fact that $Q + K^{\top} R K$ is symmetric. For parts $(b)$ and $(c)$, we first note that if $K_*$ is a strict local minimizer in $\ca S$, since $f \in C^{\omega}(\ca S)$, first-order and second-order optimality conditions imply $\nabla f(K_*) = 0$ and $\nabla^2 f(K_*) \succ 0$. By the Hessian formula in~\cite{bu2019lqr}, we have $R + B^{\top} X_* B \succ 0$, i.e., $(b)$ holds. Further, since $\nabla f(K_*) = {\bf N}_{K_*} Y_{K_*}$ and $Y_{K_*} \succ 0$, it follows that $N_{K_*} = 0$. Namely, $RK_* - B^{\top} X_* A_{K_*} = 0$. Substituting $K_* = (R+B^{\top}X_*B)^{-1}B^{\top}X_*A$ into the Lyapunov equation~\eqref{eq:lyapunov_matrix}, we have that $X_*$ solves the DARE~\eqref{eq:dare}. For part $(d)$, it suffices to observe that $K_*$ is the unique stationary point. To this end, suppose that there exist $K_{*, 1}$ and $K_{*, 2}$ such that the gradient vanishes at both points, namely ${\bf N}_{K_{*, 1}} = {\bf N}_{K_{*, 2}} = 0$\footnote{This follows from $Y_{K} \succ 0$ for every $K \in \ca S$.}. By part $(d)$ in Proposition~\ref{prop:analytical}, we have
\begin{align*}
  X_{*, 1} - X_{*, 2} &= A_{K_{*,2}}^{\top}(X_{*,1}-X_{*,2})A_{K_{*,2}} \\
  & -(K_{*,1}-K_{*,2})^{\top}(R+B^{\top}X_{*,1}B)(K_{*,1}-K_{*,2}).
  \end{align*}
  As $A_{K_{*,2}}$ is Schur, it follows that $X_{*, 1} \succeq X_{*, 2}$ and similarly $X_{*, 2} \succeq X_{*, 1}$. Hence, the stationary point is unique. Part $(e)$ follows from standard DARE theory (see Chapters $12$ and $13$ in~\cite{Lancaster1995algebraic} for details.)
  \end{proof}
\section{Global Convergence of Policy Gradient Algorithms}
\label{sec:convergence}
In this section, we show the global convergence of gradient descent (GD), natural gradient descent (NGD), and quasi-Newton (QN) iterations for indefinite least squares optimal control. In particular, under Assumption $1$, it is shown that gradient descent (respectively, natural gradient descent and quasi-Newton) converges to the maximal solution of the DARE at a linear (respectively, linear and quadratic) rate. In this direction, first recall that the gradient, natural gradient and quasi-Newton directions~\cite{bu2019lqr} are given by,
                  \begin{align*}
                     {\bf g}(K) &\coloneqq 2(R K- B^{\top} X A_{K}) Y,\\
                    {\bf n}(K) &\coloneqq 2(R K- B^{\top} X A_{K}) ,\\
                    {\bf qn}(K) &\coloneqq 2(R+B^{\top} X B)^{-1}(RK-B^{\top}XA_{K});
                    \end{align*}
                    GD, NGD and QN now refer to following update rules:
                    \begin{align}
                      \text{GD}: & &&K_{j+1} = K_j - \eta_j {\bf g}(K_j), \label{eq:gd}\\
                      \text{NGD}: & &&K_{j+1} = K_j - \eta_j {\bf n}(K_j), \label{eq:ngd}\\
                      \text{QN}: & &&K_{j+1} = K_j - \eta_j {\bf qn}(K_j),\label{eq:qn}
                      \end{align}
                      where $\eta_j$'s are stepsizes to be determined.
                  We provide the convergence analysis for the case of natural gradient descent.
                \begin{theorem}[Natural Gradient Analysis]
                  \label{thrm:ng_inner_convergence}
                 Consider the iterates $\{K_j\}$ generated by NGD~\eqref{eq:ngd}, with stepsize $\eta_j = 1/(2 \lambda_n(R+B^{\top}X_j B))$,
                  %namely,
                   %\begin{align*}
                     %K_{j+1} = K_j - \frac{1}{2\lambda_n(R_1 + B_1^{\top} X_j B_1)} 2{\bf N}_j,
                    %\end{align*}
                    where $\{X_j\}$ solve the corresponding Lyapunov equations~\eqref{eq:lyapunov_matrix}. Then both the function values and gain iterates converge to their corresponding global minima at a linear rate. That is,
                    \begin{align*}
                       f(K_j) - f(K_*) &\le q_1^j (f(K_0)-f(K_*)),\\
                      \|K_{j} - K_*\|_F^2 &\le c_1q_1^{j}\|K_0 - K_*\|_F^2,
                      \end{align*}
                      for some $q_1 \in (0, 1)$ and $c_1 > 0$.
                  \end{theorem}
                  %
                  %\begin{remark}
                    %The difference between above theorem and the standard results treated in~\cite{fazel2018global,bu2019lqr} is that we are not assuming $Q$ to be positive definite. Mind that in the above problem, the matrix $Q-L_j^{\top} R_2 L_j$ corresponding to the penalization on states in the standard LQR, can be indefinite in our sequential algorithms. The one step progression would follow from Theorem in~\cite{bu2019lqr}. However, the important difference is that the cost function is no longer coercive, requiring a separate analysis for establishing the stability of the iterates.\footnote{We note that stability of the iterates in natural gradient update was not explicitly shown in~\cite{fazel2018global}. Some of the perturbation arguments for gradient descent in~\cite{fazel2018global} can however be applied to argue for this property. Such an argument would however rely on the strict positivity of the minimum eigenvalue of $Q$.}
                    %\end{remark}
                  \begin{proof}
                    The analysis provided in~\cite{bu2019lqr} for the one-step progression of NGD holds here; thus the convergence rate would remain the same if we can prove that the iterates remain stabilizing. \par
                    By induction, it suffices to argue that with the chosen stepsize, $K_j$ is stabilizing provided that $K_{j-1}$ is. Consider the ray $\{K_t = K_{j-1} -  t {\bf n}(K_{j-1}): t \ge 0\}$. Note that by openness of $\ca S$ and continuity of eigenvalues, there is a maximal interval $[0, \zeta)$\footnote{We suppose $\zeta$ is finite; if $\zeta$ is infinite, there is nothing needed to be shown.} such that $K_{j-1} + t {\bf n}(K_{j-1}) $ is stabilizing for $t \in [0, \zeta)$ and $K_{j-1} + \zeta {\bf n}(K_{j-1})$ is marginally stabilizing. Now suppose that $\zeta \le {1}/({2\lambda_n(R_1+B_1^{\top} X_{i-1} B_1)})$; take a sequence $t_l \in [0, \zeta)$ such that $t_l \to \zeta$. Consider the sequence of value matrices $\{X_{t_l}\}$ and denote by $\ca L$ as the set of all limit points of $\{X_{t_j}\}$. Observe that $X_* \preceq X_{t_l} \preceq X_{j-1}$. By Bolzno-Weierstrass~\cite{rudin1976principles}, $\ca L$ is nonempty.\footnote{Note that it is not guaranteed that $X_{t_j}$ is convergent. The limit points are also not necessarily well-ordered in the ordering induced by the p.s.d. cone.} By continuity, any $Z \in \ca L$ solves,
                    \begin{align*}
                      Z=(A-B K_{\zeta})^{\top} Z  (A-B K_{\zeta})  + Q  + K_{\zeta}^{\top} R K_{\zeta}.
                    \end{align*}
                    Now by part $(d)$ in Proposition~\ref{prop:analytical}, we have
                    \begin{align*}
                      Z-X_* &= (A-B K_{\zeta})^{\top} (Z-X_*) (A-B K_{\zeta}) \\
                      &\quad + (K_{\zeta}-K_*)^{\top} (R + B^{\top} X_* B)(K_{\zeta}-K_*) .
                      \end{align*}
                      Suppose that $(\lambda, v)$ is the eigenvalue-eigenvector pair of $A-B K_{\zeta}$ such that $(A-B K_{\zeta}) v = \lambda v$ and $|\lambda| = 1$. Then it follows that,
                      \begin{align*}
                      v^{\top}(Z-X_*)v &= v^{\top}(A-B K_{\zeta})^{\top} (Z-X_*) (A-B K_{\zeta})v \\
                      &+ v^{\top}(K_{\zeta}-K_*)^{\top} (R + B^{\top} X_* B)(K_{\zeta}-K_*)v .
                        \end{align*}
                       Thereby $(K_{\zeta} -K_*)v = 0$ and $K_{\zeta} v = K_* v$. Consequently, $(A-BK_*) v = (A-BK_{\zeta}) v$. But this is a contradiction to the assumption that $K_*$ is a stabilizing solution.
                       
                       Hence $\{X_j\}$ is a monotonically non-increasing sequence bounded below by $X_{*}$. As such, the sequence of iterates $\{K_j\}$ and the sequence of function values $\{f(K_j)\}$ converge linearly to $K_*$ and $f(K_*)$ following the arguments in~\cite{bu2019lqr}.
                    \end{proof}
                    We mention that the above stability argument can be applied for the sequence generated by the quasi-Newton iteration as well. The quadratic convergence rate for such a sequence would then follow from the proof in~\cite{bu2019lqr}.
                    \begin{theorem}[Quasi-Newton Analysis]
                  \label{thrm:qn_inner_convergence}
                  Suppose Assumption $1$ holds. Consider the iterates $\{K_j\}$ generated by QN~\eqref{eq:qn} with stepsize $\eta_j = 1/2$.
                    Then both the function values and iterates converge to their respective global minima at a $Q$-quadratic rate. That is,
                    \begin{align*}
                       f(K_j) - f(K_*) &\le q_2 (f(K_{j-1})-f(K_*))^2,\\
                      \|K_{j} - K_*\|_F^2 &\le c_2q_2 \|K_{j-1} - K_*\|_F^4,
                      \end{align*}
                      for some $q_2 > 0$ and $c_2 > 0$.
                   %Suppose Assumption $1$ holds. Consider with constant stepsize $\eta_j = 1/2$, the iterates $\{K_j\}$ generated by QN
                    %converges quadratically to $K_*$ provided $K_0 \in \ca S$. That is,
                      %for some $q > 0$,
                  \end{theorem}
                  The gradient policy analysis requires more work since the stepsize developed in~\cite{bu2019lqr} involves the smallest eigenvalue $\lambda_1(Q)$. However by carefully replacing ``$\lambda_1(Q)$-related quantities'' in~\cite{bu2019lqr}, one can still prove the global linear convergence rate as follows.
                    \begin{theorem}[Gradient Analysis]
                  \label{thrm:gd_inner_convergence}
                  Suppose Assumption $1$ holds. Consider the iterate $\{K_j\}$ generated by GD~\eqref{eq:gd} with stepsize $\eta_j$ {specified in Theorem~\ref{thrm:gd_step}.}
                     Then both the function values and iterates converge to their respective global minima at a linear rate. That is,
                    \begin{align*}
                       f(K_j) - f(K_*) &\le q_3^j (f(K_0)-f(K_*)),\\
                      \|K_{j} - K_*\|_F^2 &\le c_3q_3^{j}\|K_0 - K_*\|_F^2,
                      \end{align*}
                      for some $q_3 \in (0, 1)$ and $c_3 > 0$.
                   %\begin{align*}
                     %K_{j+1} = K_j - \eta_j {\bf g}(K_j),
                    %\end{align*}
                    %converges linearly to $K^+$ provided $M_0 \in \ca S$. That is,
                    %\begin{align*}
                      %\|K_{i} - K_*\|_F^2 \le q^i \|K_0 - K_*\|_F^2,
                      %\end{align*}
                      %for some $q \in (0,1)$.
                  \end{theorem}
                  %The convergence analysis of gradient policy follows closely of the idea presented in~\cite{bu2019lqr}.
                  In~\cite{bu2019lqr}, the compactness of sublevel sets have been used to devise the stepsize rule to guarantee a sufficient decrease in the cost and stability of the iterates. The proof of compactness in~\cite{bu2019lqr} however, relies on the positive definiteness of $Q$ and $R$.\footnote{Or the observability of $(Q, A)$.} But, we can show that a perturbation bound can be employed to derive a suitable constant stepsize for the indefinite cost structure as well. The details of this observation are deferred to the {Appendix~\ref{sec:gd_inner}.}
                  \section{A Numerical Example}
                  \label{sec:numerical}
                  In this section, we show the proposed convergence results by a numerical example. The system parameters are $A=0.5 I$, $B=I$, $R=I$ and
                  \begin{equation*}
                  \resizebox{\linewidth}{!}{
                    $Q = \begin{pmatrix} 1.62370842 &  0.36712592 & -1.31209102 &  1.97803823 & -0.49297266 \\
        0.36712592 &  2.21878741 & 0.47525552 & -1.07142839 & 1.04343275 \\
       -1.31209102 &  0.47525552 & 1.90887732 &-0.83057818 & 0.3818043 \\
        1.97803823 & -1.07142839 &-0.83057818 & 0.93847322 &-0.90779531\\
       -0.49297266 &  1.04343275 & 0.3818043  &-0.90779531 & -1.06295748\\
                      \end{pmatrix}.$}
                    \end{equation*}
                    Note that $Q$ is indefinite and its (rounded) eigenvalues are $4.75, 2.55, 0.96, -1.1, -1.53$.
                    Figures~\ref{fig:fig1}-\ref{fig:fig2} show the global linear convergence of the gradient policy update. The global linear convergence of natural gradient policy are demonstrated in Figures~\ref{fig:fig1_ngd}-\ref{fig:fig2_ngd}. Figures~\ref{fig:fig1_qn}-\ref{fig:fig2_qn} show the $Q$-quadratic convergence for the quasi-Newton policy update.
\begin{figure}[ht]
    \centering
    \begin{minipage}{0.48\linewidth}
        \centering
     \includegraphics[width=0.98\textwidth]{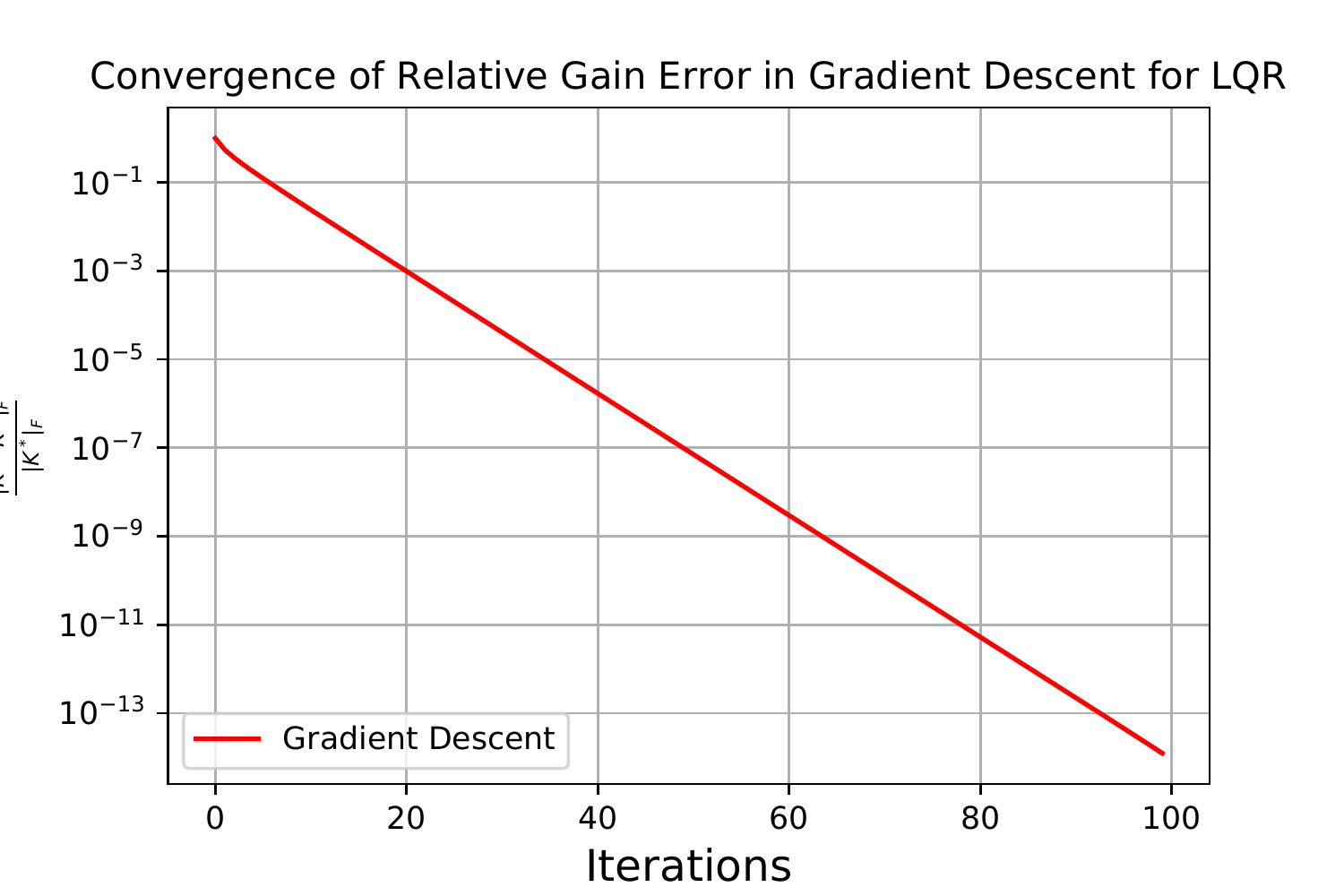}
      \caption{\footnotesize Convergence of the relative error for the feedback gain under gradient descent iteration}\label{fig:fig1}
    \end{minipage}\hfill
    \begin{minipage}{0.48\linewidth}
        \centering
        \includegraphics[width=0.98\textwidth]{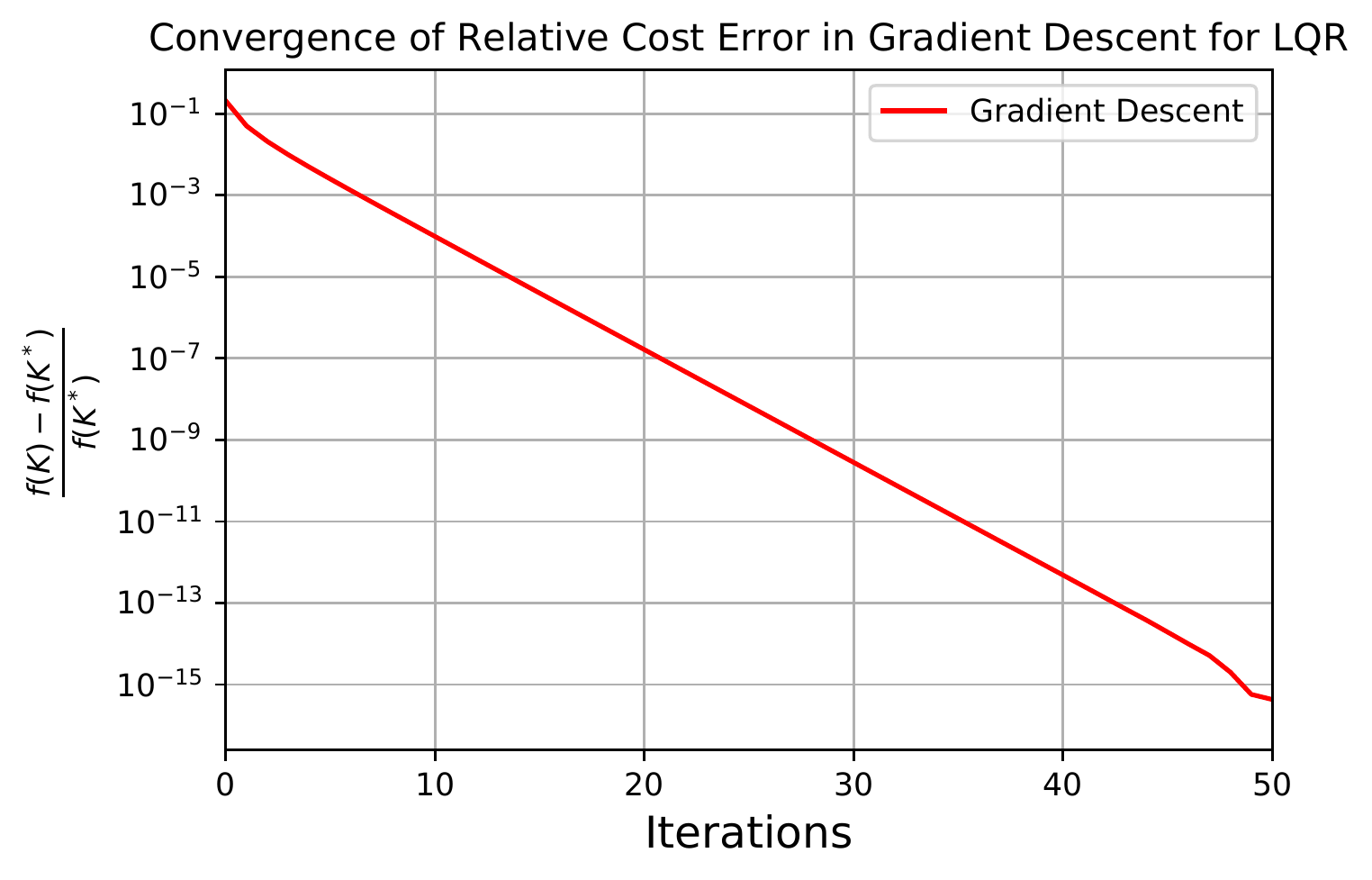}        
        \caption{\footnotesize Convergence of the relative error for indefinite LQR cost under gradient descent iteration}\label{fig:fig2}
    \end{minipage}
\end{figure}
                    \begin{figure}[ht]
    \centering
    \begin{minipage}{0.48\linewidth}
        \centering
     \includegraphics[width=0.98\textwidth]{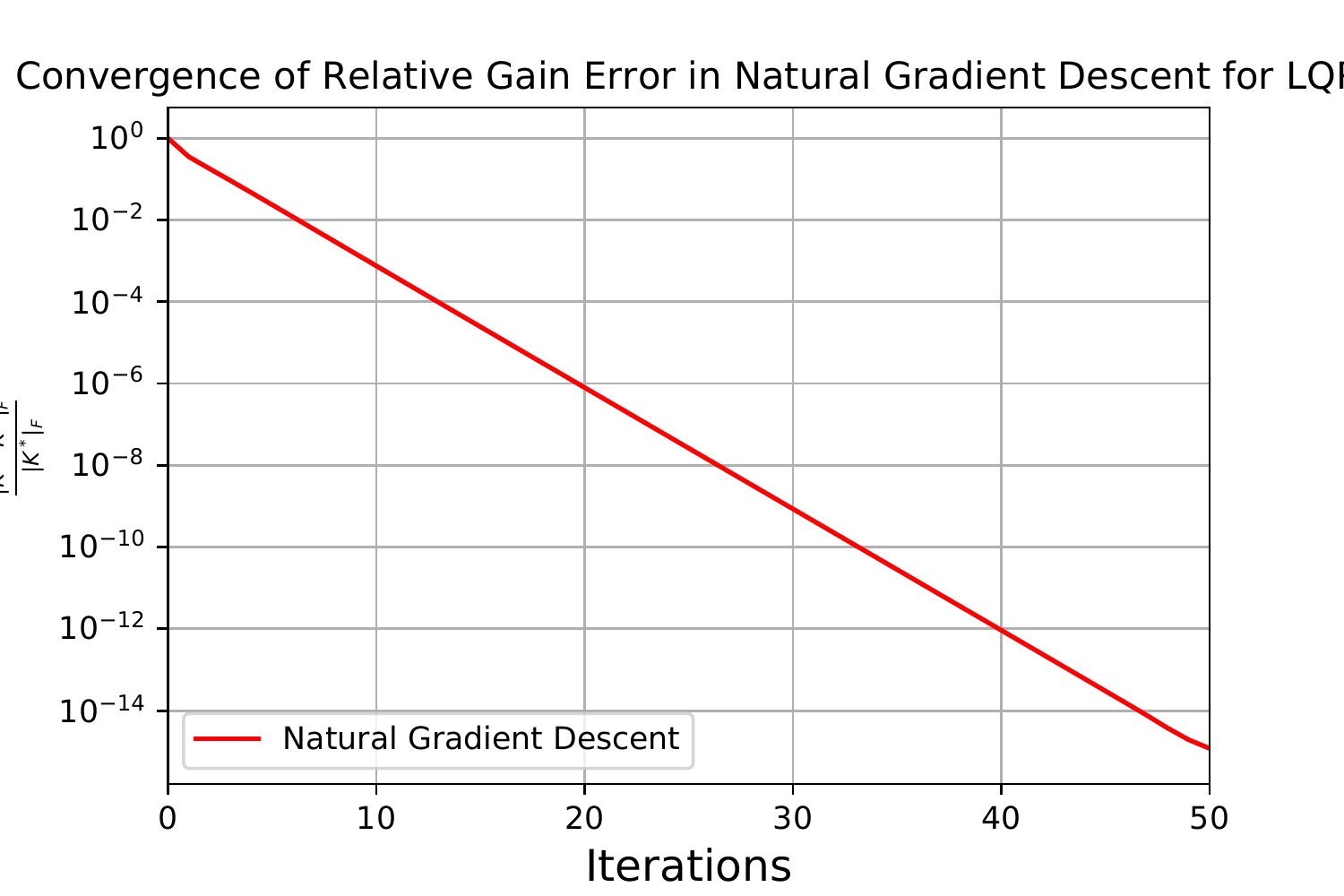}
      \caption{\footnotesize Convergence of the relative error for the feedback gain under natural gradient descent iteration}\label{fig:fig1_ngd}
    \end{minipage}\hfill
    \begin{minipage}{0.48\linewidth}
        \centering
        \includegraphics[width=0.98\textwidth]{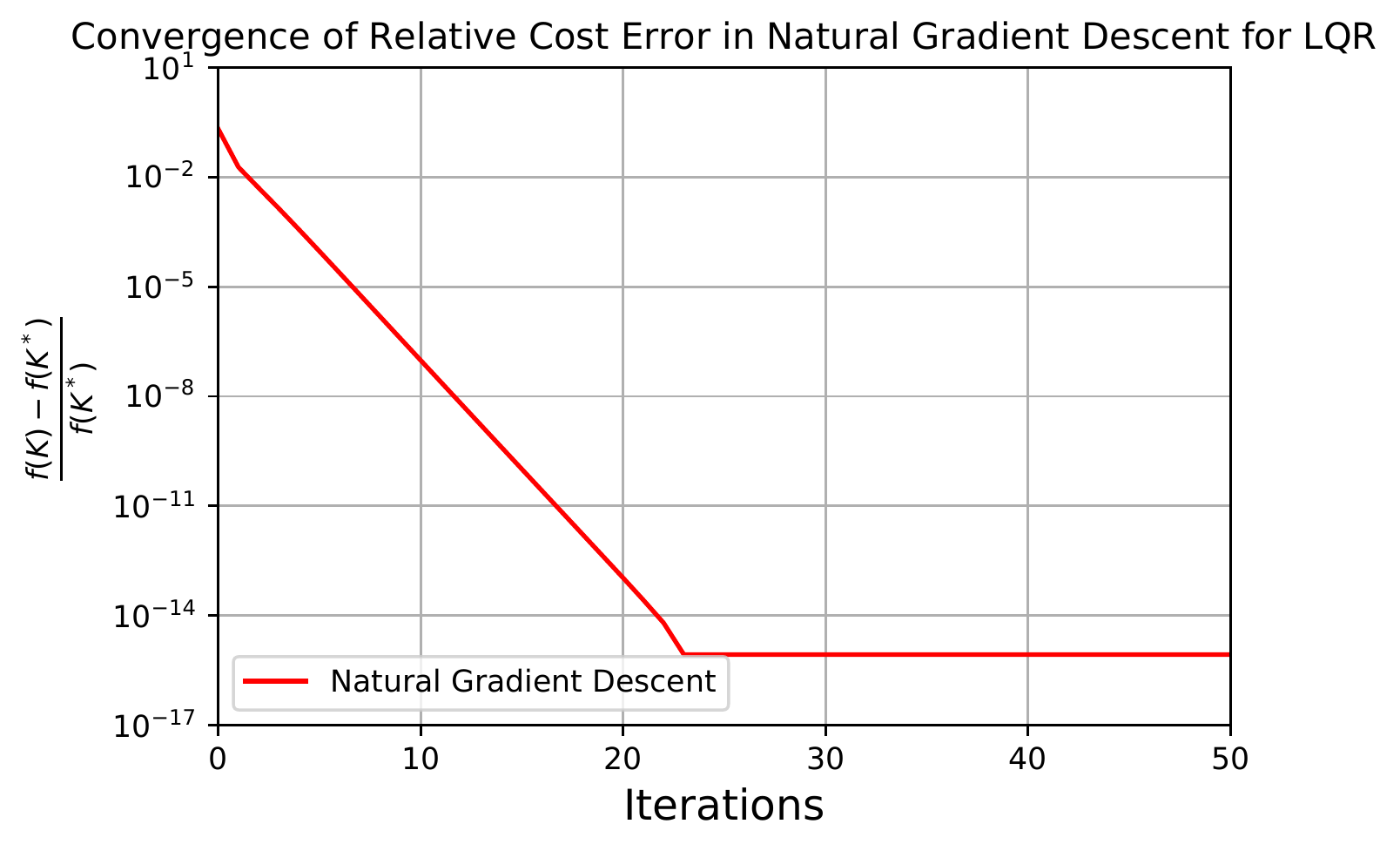}        
        \caption{\footnotesize Convergence of the relative error for indefinite LQR cost under natural gradient descent iteration}\label{fig:fig2_ngd}
    \end{minipage}
\end{figure}
\begin{figure}[ht]
    \centering
    \begin{minipage}{0.48\linewidth}
        \centering
     \includegraphics[width=0.98\textwidth]{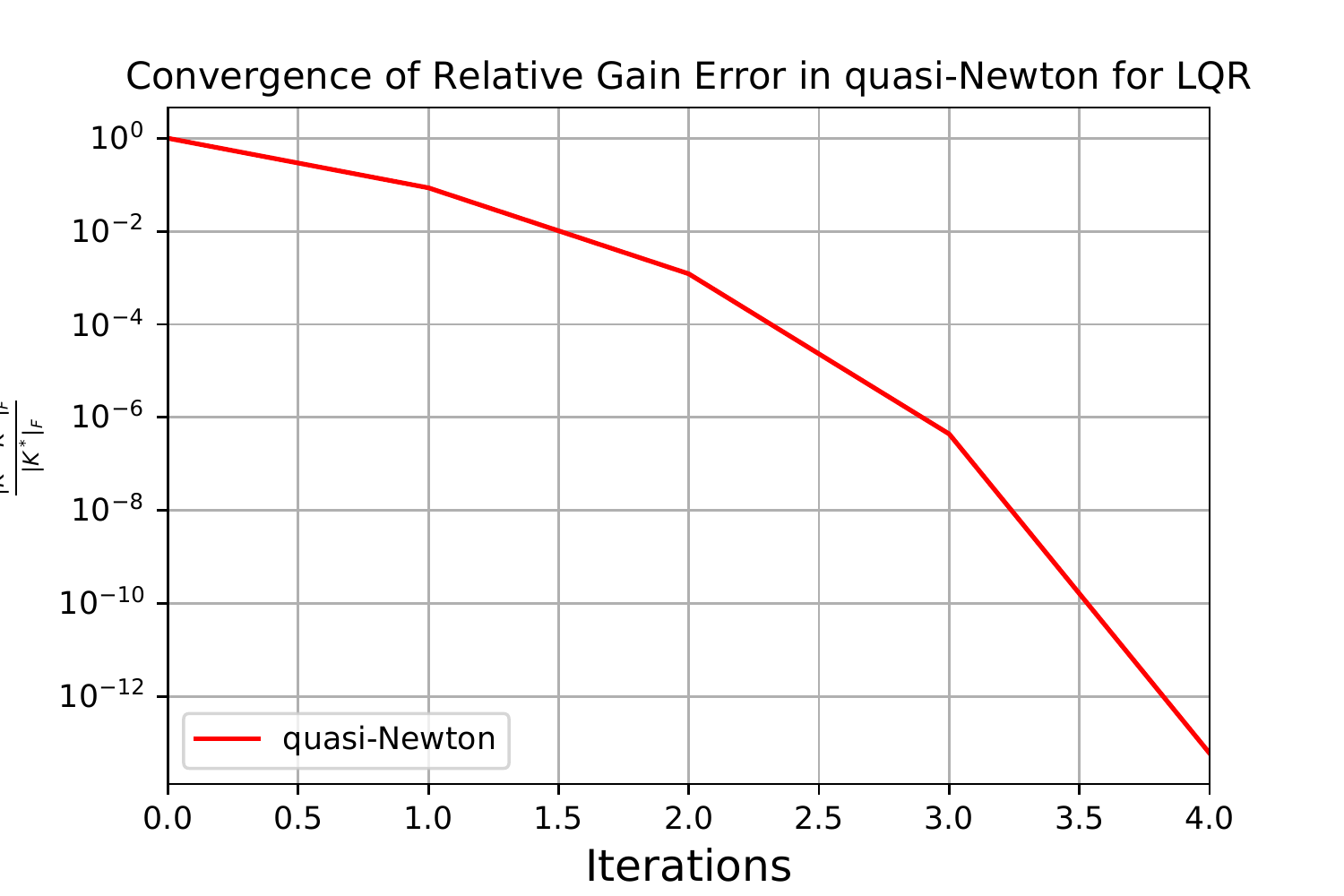}
      \caption{\footnotesize Convergence of the relative error for the feedback gain under quasi-Newton iteration with constant stepsize $1/2$ } \label{fig:fig1_qn}
    \end{minipage}\hfill
    \begin{minipage}{0.48\linewidth}
        \centering
        \includegraphics[width=0.98\textwidth]{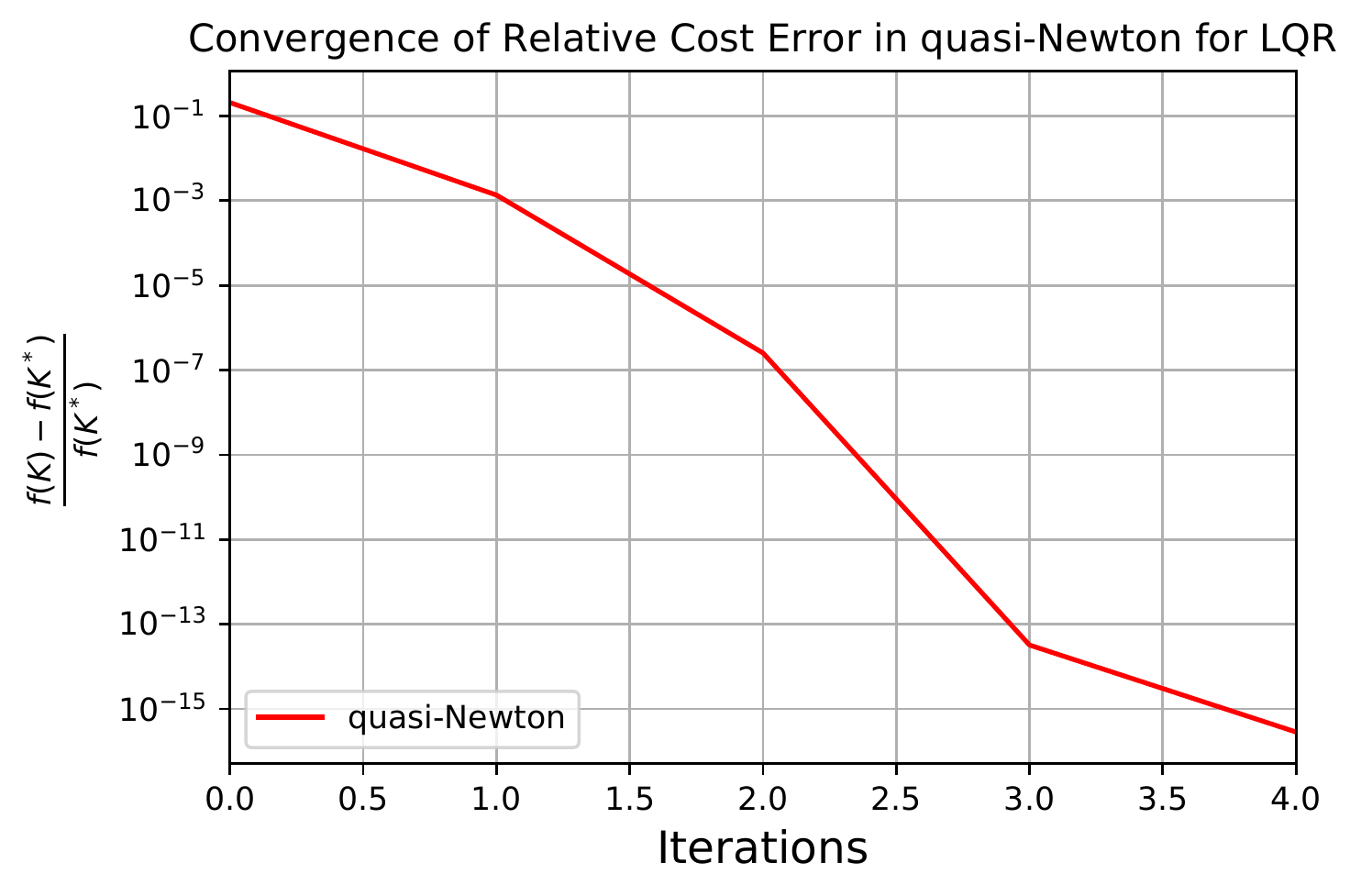}        
         \caption{\footnotesize Convergence of the relative error for indefinite LQR cost under quasi-Newton iteration with constant stepsize $1/2$}
        \label{fig:fig2_qn}
    \end{minipage}
\end{figure}

                  \section{Concluding Remarks}
                  \label{sec:conclusion}
                  This note considers policy gradient algorithms for the indefinite least squares stationary optimal control, e.g., indefinite LQR. We show the global linear (respectively, linear and $Q$-quadratic) convergence of gradient policy (respectively, natural gradient and quasi-Newton policies.) Although these results are presented assuming the knowledge of the system matrices,  gradient and natural gradient policies can be extended to model-free case by means of stochastic (zeroth order) optimization (see \cite{fazel2018global} for details). As such, this note extends the results reported in~\cite{fazel2018global, bu2019lqr} for indefinite LQR. These extensions have important implications for  optimal control, stability analysis and LQ games. Indeed, some of these observations have been utilized to show global convergence of sequential policy updates in LQ dynamic games~\cite{bu2019global}.
                  \section*{Acknowledgements}
                  The authors thank Henk J. van Waarde for many helpful discussions.
\bibliographystyle{ieeetran}
\bibliography{ref}

\appendix
\counterwithin{theorem}{subsection}
\subsection{Gradient Policy Analysis for Nonstandard LQR}
    \label{sec:gd_inner}
    This section is devoted to the proof of Theorem~\ref{thrm:gd_inner_convergence}. As it was pointed out previously, the strategy adopted in~\cite{bu2019lqr, fazel2018global} are no longer viable for an indefinite cost structure.
    % for devising a stepsize guaranteeing linear convergence is similar to the one presented in~\cite{bu2019lqr}. However, the convergence analysis for the game setup is more involved as one can not estimate the needed quantities using the current function values due to the indefiniteness of the term $Q$.
    However, as we will show, a perturbation bound would circumvent this issue and allows deriving the required stepsize, guaranteeing a decrease in function values while ensuring stabilization.\par
    In the following, we shall drop all the subscripts as the stepsize will be valid for every iterate. Suppose now that we have a stabilizing policy $K$ and the gradient direction is given by ${\bf g}(K) = 2{\bf N}Y$.\footnote{Note the subscripts are dropped; ${\bf N}$ and $Y$ are both dependent on $K$} The main object that we work with in this section is the ray starting at $K$ along the gradient direction,
    $$\{K_{\eta}: K - \eta {\bf g}(K), \eta \ge 0\}.$$
    We shall further denote $A_{\eta} = A - BK_{\eta} = A-B(K - \eta 2 {\bf N} Y)$. \par
    Here is an outline of our proof strategy:
    \begin{enumerate}
    \item By the openness of $\ca S$ and continuity of eigenvalues, there exists a maximal interval $[0, c)$ such that $K_{\eta}$ is stabilizing for every $\eta < c$ and $K_c$ is marginally stabilizing; such a $c$ could be either finite or infinite.
    \item Now suppose that $c$ above is known. Then for every $\eta < c$, $f(K_{\eta})$ is well-defined and we can compute the difference,
    \begin{align*}
      f(K) - f(K_\eta) = 4\eta \Tr\left( {\bf N}^{\top} {\bf N} (Y Y_\eta - \eta a Y Y_\eta Y) \right),
    \end{align*}
    where $a = \lambda_n(R + B^{\top}X B)$, and $Y_\eta$ solves the Lyapunov matrix equation,
    \begin{align*}
      (A-B K_{\eta}) Y_\eta (A-B K_{\eta})^{\top} + {\bf \Sigma} = Y_\eta.
    \end{align*}
    \item Next we define a univariate function $\phi: [0, c) \to \bb R$ by,
    \begin{align*}
      \phi(\eta) = \Tr\left( {\bf N}^{\top} {\bf N} (Y Y_\eta - \eta a Y Y_{\eta} Y) \right).
    \end{align*}
    Note that $\phi(0) > 0$ if the gradient does not vanish at $K$. Now our goal is to characterize a step size $0 < \eta' < c$ such that $\phi(\eta') > 0$.
      \end{enumerate}
It is clear that the knowledge of $c$ and characterizing $\eta'$ above are crucial for stepsize analysis.
%Apparently, explicitly computing the analytical form of $c$ seems rather infeasible. 
We shall demonstrate that characterizing $\eta'$ will suffice to provide a stepsize; the quadratic cost structure will implicitly enforce stabilization.
    %If $K_\eta = A - B (K - \eta 2 {\bf N} Y)$ is stabilizing, then
    %Now define a univariate function $\phi$ as,
    %\begin{align*}
      %\phi(\eta) = \Tr\left( {\bf N}^{\top} {\bf N} (Y Y_\eta - \eta a Y Y_{\eta} Y) \right).
    %\end{align*}
    %We observe that $\phi(\eta)$ is well-defined locally around $0$ by openness of the set $\ca S_{L}$.\footnote{$\phi$ is well-defined only if $A- B K_{\eta}$ is Schur.}
%%Denote $[0, \eta_0)$ the largest half open interval such that $\phi$ is defined. Namely, $M_\eta$ is stabilizing for $\eta < \eta_0$ and $M_{\eta_0}$ is marginally stabilizing by continuity. 
      %Further, observe that $\phi(0) > 0$ if the gradient does not vanish at $K$. Now our goal is to characterize a step size such that $\phi(\eta) > 0$.

      To begin, we observe a perturbation bound on $Y_\eta$, assuming that $K_{\eta}$ is stabilizing.
    \begin{proposition}
Put $\mu_1 = \|Y\|_2  \|B_1 {\bf N} Y\|_2^2/\lambda_1({\bf \Sigma})$ and $\mu_2 = \|Y\|_2 \|B_1 {\bf N}Y\|_2\|A_K\|_2/\lambda_1({\bf \Sigma}) $, and let
\begin{align*}
              \eta_0 = \frac{\sqrt{\mu_1 + \mu_2^2}}{4\mu_1} - \frac{\mu_2}{4\mu_1};
  \end{align*} suppose that
      $A_{\eta}$ is Schur stable for every $\eta \le \eta_0$. Then for all $\eta \le \eta_0$,
      \begin{align*}
        \|Y_\eta\|_2 \le \beta_0 \|Y\|_2,
        \end{align*}
        where $\beta_0 = {1}/({1- 4 \mu_1 \eta_0^2 - 4\mu_2 \eta_0}) > 0$.
      \end{proposition}
      \begin{proof}
        Taking the difference of the corresponding Lyapunov equations, we have
        \begin{align*}
          &Y_\eta -Y - A_K (Y_\eta-Y) A_K^{\top} \\
          &= 2 \eta \left(A_K Y_\eta  (B_1{\bf N} Y)^{\top} + B{\bf N} Y Y_\eta A_K^{\top} \right) \\
          &\quad + 4\eta^2 B{\bf N}Y Y_\eta (B {\bf N} Y)^{\top}\\
                                                         &\preceq \|Y_{\eta}\|_2 \left( 4\eta \|B {\bf N} Y\|_2\|A_K\|_2 + 4\eta^2 \|B {\bf N}Y\|_2^2 \right) I \\
          &\preceq \|Y_{\eta}\|_2 \left( 4\eta \|B {\bf N} Y\|_2\|A_K\|_2 + 4\eta^2 \|B {\bf N}Y\|_2^2 \right) \frac{{\bf \Sigma}}{\lambda_1({\bf \Sigma})}.
        \end{align*}
        It thus follows that,
        \begin{align*}
          Y_{\eta} - Y \preceq \frac{\|Y_{\eta}\|_2 \left( 4\eta \|B {\bf N} Y\|_2\|A_K\|_2 + 4\eta^2 \|B {\bf N}Y\|_2^2 \right) }{\lambda_1({\bf \Sigma})} Y.
          \end{align*}
        Hence,
          \begin{equation*}
            \resizebox{\linewidth}{!}{$
              \|Y_{\eta}\|_2 \left( 1- \frac{\|Y\|_2 \left( 4\eta \|B {\bf N} Y\|_2\|A_K\|_2 + 4\eta^2 \|B {\bf N}Y\|_2^2 \right) } {\lambda_1({\bf \Sigma})} \right)\ \le \|Y\|_2.
              $}
            \end{equation*}
            The proof is completed by a direct computation showing that $1/\beta_0 = 1-\mu_1 \eta_0^2 - 4 \mu_2 \eta_0 > 0$ with the choice of $\eta_0$ and for every $\eta \le \eta_0$, 
            \begin{align*}
              1-4\mu_1 \eta^2 - 4 \mu_2 \eta \ge 1-4\mu_1 \eta_0^2 - 4 \mu_2 \eta_0.
            \end{align*}
        \end{proof}
        %We note if $\eta' \le \eta_0$ and $A_{\eta'}$ is Schur stable, then $\|Y_{\eta'}\|_2 \le \beta_{\eta'}\|Y\|_2 \le \beta_{\eta_0} \|Y\|_2$.\\
        %We now present an important result for our analysis. The basic idea of this lemma is as follows: if $[0,c)$ is the largest interval such that $A_{t}$ is Schur stable for every $t \in [0, c)$ and $A_c$ is marginally Schur stable,\footnote{Such a $c$ exists by the openness of the set of stabilizing gains.} then we can find a number $c_0 < c$ such that $f(K_s) \le f(K)$ for every $s \in [0, c_0]$.
        The next lemma shows that if $c$ is known, a positive stepsize can be chosen.
    \begin{lemma}
      Let $c$ be the largest real positive number such that $A_{t}$ is Schur stable for every $t \in [0, c)$ and $A_c$ is marginally Schur stable.\footnote{Here we have assumed that $c$ is not $+\infty$. Of course, if $c=+\infty$, then any stepsize would lead to a stabilizing update.} Let
\begin{align*}
  a_1 = a \beta_0 \|Y\|_2 + 4\|{\bf N}\|_2 \beta_0 \|Y\|_2^2,\; a_2 &= a 4\|{\bf N}\|_2 \beta_0 \|Y\|_2^2;
  \end{align*}
then with $\eta_1 \le \min(c-\varepsilon, \eta_0, c_0)$,
        where $\varepsilon > 0$ is an arbitrary positive real number and
        \begin{align*}
          c_0 < \sqrt{\frac{ 1}{a_2} + \frac{a_1^2}{4a_2^2}} - \frac{a_1}{2a_2},
          \end{align*}
        one has $\phi(\eta_1) \ge 0$.
      \end{lemma}
      \begin{proof}
        The computation follows a similar method used in~\cite{bu2019lqr} by replacing the estimate of $Y(\theta)$ by the bound in the above proposition (see details in Lemma $5.5$ in~\cite{bu2019lqr}). 
        %Putting $b = $ and $c = a 4\|{\bf U}\|_2 \beta_0 [\lambda_n(Y)]^2$.
        \end{proof}
        % If one could explicitly compute the $c$ in the above result, then a deterministic choice of stepsize could be chosen; however, this is not feasible. Fortunately,
        Finally, we show that $c > \min(\eta_0, c_0)$. This would then imply that one can choose the stepsize as $\eta = \min(\eta_0, c_0)$.
        \begin{theorem}
          \label{thrm:gd_step}
          With the stepsize $\eta = \min(\eta_0, c_0)$,  $M_{\eta}$ remains stabilizing and $\phi(\eta) \ge 0$.
          \end{theorem}
\begin{proof}
 Let $\eta = \min(\eta_0,c_0)$. It suffices to prove that for every $t \in [0, \eta]$, $A_t$ is Schur stabilizing and $\phi(t) \ge 0$. We prove this by contradiction. Suppose that this is not the case. Then by continuity of eigenvalues, there exists a number $\eta' \le \eta$ such that $A_s$ is stabilizing for every $s \in [0, \eta')$ and $K_{\eta'}$ is marginally stabilizing. If this is the case, the choice of $\eta_0, c_0$ guarantees that for every $s \in [0, \eta')$, $\phi(s)$ is well-defined and $\phi(s) \ge 0$.
  %\begin{enumerate}
    %\item $M_{\eta'}$ is stabilizing but $M_{\eta'}$ leaves the sublevel set.
    %\item $M_{\eta'}$ is not stabilizing.
    %\end{enumerate}
                          %We shall prove that both scenarios can not occur by establishing a contradiction.
                          %Suppose that scenario $(a)$ holds; note that the function $t \mapsto h(M - t \nabla h(M)) \eqqcolon \phi(t)$ is well-defined and continuous over $[0, \eta']$. Also observe that for sufficiently small $t>0$, we have $\phi(t) < \phi(0)$. Thus by the Intermediate Value Theorem, there must be some $\zeta < \eta'$ for which $h(M - \zeta \nabla h(M)) = h(M)$, establishing is a contradiction to the strict inequality. \\
                         % Now suppose that scenario $(b)$ holds; by continuity of eigenvalues, we can assume that there exists some $\zeta \le r$ for which $M_{i+1}(t) \coloneqq M_i - t \nabla_K h(M_i)$ is stabilizing while $M_{i+1}(\zeta)$ is marginally stabilizing. By the argument presented under the first scenario, $M_{i+1}(t)$ is contained in the sublevel set for every $t < \zeta$. 
Now take a sequence $t_i \to \eta'$ and consider the corresponding sequence of value matrices $\{X_{t_i}\}$. Note that the sequence of function values $\Tr(X_{t_i} {\bf \Sigma})$ satisfies,
\begin{align*}
  \Tr(X_*{\bf \Sigma}) \le \Tr(X_{t_i}{\bf \Sigma}) \le \Tr(X{\bf \Sigma}),
  \end{align*}
  since $\phi(t) \ge 0$.
But this implies that $\{X_{t_i}\}$ is a bounded sequence (note that the above inequality on function values does not guarantee the boundedness of the sequence; it is crucial that $X_{t_i} \succeq X_*$). 
%The sequence is bounded below by $X^+$ and also bounded above since,
                          %\begin{align*}
                            %\|X_{t_i}\|_2 \Tr({\bf \Sigma}) \le \Tr\left( X_{t_i} {\bf \Sigma} \right) \le \Tr \left( X_{M_i} {\bf \Sigma}\right).
                            %\end{align*}
                            Hence by a similar argument adopted in the proof of Theorem~\ref{thrm:ng_inner_convergence}, these observations establish
                             a contradiction; as such, the proposed stepsize guarantees stabilization.
  \end{proof}
  It is now straightforward to conclude the convergence rate of Theorem~\ref{thrm:gd_inner_convergence} by similar arguments as in~\cite{bu2019lqr}.\footnote{Strictly speaking, we need to show our proposed stepsizes are bounded away from $0$. Namely, that there is some constant $d > 0$ such that $\eta_j > d$ for every $j$. The computations are omitted here due to space limitation. In the meantime, one can be convinced of this fact by checking the asymptotics of $\eta_0$ and $c_0$.}

%%%%%%%%%%%%%%%%%%%%%%%%%%%%%%%%%%%%%%%%%%%%%%%%%%%%%%%%%%%%%%%%%%%%%%%%%%%%%%%%

\end{document}